\documentclass[12pt, a4paper]{amsart}

\usepackage{fullpage}
\usepackage[dvipsnames,table]{xcolor}
\usepackage{colortbl}
\definecolor{webcolor}{rgb}{0,0,1}
\definecolor{webbrown}{rgb}{.6,0,0}
\usepackage[
        colorlinks,
        linkcolor=webbrown, filecolor=webbrown,citecolor=webbrown,
        urlcolor=webbrown,
        pdfauthor={},
  pdftitle={},
]{hyperref}

\usepackage[utf8]{inputenc}
\usepackage{amsmath}
\usepackage{amsthm}
\usepackage{amssymb}
\usepackage[capitalise]{cleveref}
\usepackage{booktabs}
\usepackage[normalem]{ulem}
\usepackage{comment}
\usepackage{fvextra}

\usepackage{tikz}
\usetikzlibrary{arrows.meta,positioning,calc}
\usepackage{tikz-cd}

\definecolor{accent}{RGB}{200,80,20}

\def\R{\mathbb R}

\def\Z{\mathbb Z}

\def\Q{\mathbb Q}

\def\Vcv{V_{\mathrm{cv}}}

\def\II{\mathrm{II}}

\def\RR{\mathbb R}
\def\NN{\mathbb N}

\def\ZZ{\mathbb Z}
\def\FF{\mathbb F}
\def\QQ{\mathbb Q}
\def\LAT{\mathsf{Lat}}
\def\WLAT{\mathsf{WLat}}
\def\SET{\mathsf{Set}}
\def\GRAPH{\mathsf{Graph}}

\DeclareMathOperator{\Gcan}{\mathit{G}_{\mathrm{can}}}
\DeclareMathOperator{\sv}{sv}
\DeclareMathOperator{\cv}{cv}

\DeclareMathOperator{\rank}{rank}
\DeclareMathOperator{\GL}{GL}

\DeclareMathOperator{\Aut}{Aut}
\DeclareMathOperator{\Autred}{Aut_{\mathrm{red}}}

\DeclareMathOperator{\Smin}{\mathit{S}_{\mathrm{min}}}
\def\id{\mathrm{id}}

\newtheorem{theorem}{Theorem}[section]

\newtheorem{prop}[theorem]{Proposition}
\newtheorem{lemma}[theorem]{Lemma}
\theoremstyle{definition}
\newtheorem{definition}[theorem]{Definition}
\newtheorem{example}[theorem]{Example}
\newtheorem{remark}[theorem]{Remark}

\newtheorem*{problem*}{Problem}

\Crefname{cor}{Corollary}{Corollaries}
\Crefname{lemma}{Lemma}{Lemmas}
\Crefname{example}{Example}{Examples}
\Crefname{prop}{Proposition}{Propositions}
\AddToHook{env/cor/begin}{\crefalias{theorem}{cor}}
\AddToHook{env/prop/begin}{\crefalias{theorem}{prop}}
\AddToHook{env/lemma/begin}{\crefalias{theorem}{lemma}}
\AddToHook{env/definition/begin}{\crefalias{theorem}{definition}}
\AddToHook{env/example/begin}{\crefalias{theorem}{example}}
\AddToHook{env/remark/begin}{\crefalias{theorem}{remark}}

\title{Automorphisms and canonical forms of integer lattices}

\author{Simon Brandhorst}
\address{
Fakult\"at für Mathematik und Informatik\\
Universit\"at des Saarlandes\\
Campus Saarbrücken \\
Gebäude E2 4 \\
66123 Saarbrücken
- Germany - 
}
\email{brandhorst@math.uni-sb.de}

\author{Tommy Hofmann}
\address{
Naturwissenschaftlich-Technische Fakult\"at\\
Universit\"at Siegen\\
Walter-Flex-Straße 3\\
57068 Siegen\\
Germany}
\email{tommy.hofmann@uni-siegen.de}

\date{\today}



\begin{document}

\renewcommand{\theenumi}{(\roman{enumi})}%

\begin{abstract}
We improve the Plesken--Souvignier algorithm for computing automorphism
groups and isometries of positive definite integer quadratic lattices by exploiting
characteristic sublattices and refining vector enumeration. For lattices
with roots, we also use the Weyl group to accelerate the computation of
canonical forms. Experiments with concrete implementations demonstrate
substantial speedups and show that canonical forms can be computed at a
cost comparable to that of computing automorphism groups.
\end{abstract}
\maketitle

\section{Introduction}
An integer quadratic lattice consists of a finitely generated free abelian group $L\cong \ZZ^r$ and a quadratic form $q\colon L \to \QQ$. Such lattices appear in various contexts such as topology, algebraic geometry, number theory and cryptography.

We assume that the quadratic form $q$ is positive definite, that is, $q(x) \geq 0$ for all $x \in L$ and $q(x) = 0$ if and only if $x = 0$. In particular, this implies that the orthogonal group
$O(L) = \{f \in \mathop{\mathrm{GL}}(L) \mid q \circ f = q\}$ is finite.
Let $b(x,y)=q(x+y)-q(x)-q(y)$ denote the bilinear form associated with $q$.
Given an ordered basis $B=(e_1,\dots, e_r)$ of $L$, we call $G_B = (b(e_i,e_j))_{1 \leq i, j \leq r}$ its Gram matrix. A canonical form is a function $\Gcan$ that associates to any lattice $L$ a Gram matrix $\Gcan(L)$ of $L$ such that two lattices $L,L'$ are isometric if and only if $\Gcan(L)=\Gcan(L')$.
In this work, we improve algorithms for the following three computational problems:
\begin{enumerate}
\item the computation of generators and the order of the orthogonal group $O(L)$;
\item the determination of whether two given lattices $L,L'$ are isometric and, if so, the computation of an explicit isometry $L \to L'$;
\item the computation of a canonical form $\Gcan$.  
\end{enumerate}

Besides the intrinsic interest of these problems, our motivation stems from the computation of the class number of a lattice:
Two lattices $L,L'$ are in the same genus if $L \otimes \R \cong L' \otimes \R$ and $L\otimes \ZZ_p \cong L' \otimes \ZZ_p$ for all rational primes $p$. The number of isometry classes in a genus is finite and is called the class number. It can be computed by Kneser's neighbor method (see the survey \cite{MR4675113}), which rests on solutions of the computational challenges (i)--(iii).
For recent applications and examples of such computations, see the works of the authors~\cite{MR4608887}, Chenevier~\cite{Chenevier2025} and Elkies~\cite{Elkies2026}.

All currently known algorithms addressing any of the three problems require finding a shortest vector of $L$ (or its dual), a problem that is known to be NP-hard under randomized reductions by work of Ajtai~\cite{Ajtai:1998} and can be solved in time $2^{O(n)}$. Haviv--Regev~\cite{Haviv-Regev2014} give $n^{O(n)}$-time and Aggarwal--Jiang--Li--Liu~\cite{Aggarwal-Kaijie-Zihan-Yichen2026} give $n^{n+o(n)}$-time solutions for the three problems (see also~\cite{Ducas-vanWoerden2022}).
These theoretical bounds are far beyond anything practical.
However, for most interesting lattices (of dimension $\leq 32$) coming from theoretical applications, shortest vectors can be efficiently enumerated using the algorithms of Fincke--Pohst \cite{FinckePohst1985} or Kannan~\cite{Kannan1983}; see also \cite{Micciancio2015}.
We therefore assess practical performance by measuring running times on datasets of lattices. 
To facilitate future comparisons of running times, we make these datasets available as a versioned benchmark set~\cite{LatticeAutBench}.

\subsection*{Orthogonal groups and isometry tests}
For about 30 years, the original algorithm of Plesken and Souvignier (1997) \cite{plesken-souvignier:1997} was the fastest published algorithm for isometry and automorphism computations of integer quadratic lattices.
Between 2020 and 2026, Chenevier \cite{Chenevier2020,Chenevier2025} and Ta\"ibi \cite{Taibi2024} substantially improved it for lattices with many roots,
which was crucial for determining the class number of unimodular lattices of ranks 28 and 29. There is also the closed-source 2009 Magma~\cite{MR1484478} implementation of Unger, which, according to the documentation, works ``with ordered partition methods due to Leon''.

Our main contribution is an improved version of the Plesken--Souvignier algorithm, which exploits characteristic sublattices to construct $O(L)$-invariant orthogonal decompositions of $L \otimes \Q$.
We also show how such decompositions can be used to reduce the cost of vector enumeration.
Although nontrivial invariant decompositions are not available for every
lattice, they occur for most lattices in our datasets. Our experiments show
significant speedups over existing implementations.
The running times are reported in \Cref{tab:automorphism-group-timings} (see \S{}\ref{sec:exp} for more details).

\begin{table}[htbp]
  \caption{Running times in seconds for the automorphism group computation on
    the datasets (A)--(C) of \cite{LatticeAutBench} described in \S{}\ref{subsec:datasets}.}
  \label{tab:automorphism-group-timings}

  \centering
  \setlength{\tabcolsep}{4.5pt}
  \begin{tabular}{l*{9}{r}}
    \toprule
    & \multicolumn{3}{c}{(A)}
    & \multicolumn{3}{c}{(B)}
    & \multicolumn{3}{c}{(C)} \\
    \cmidrule(lr){2-4}
    \cmidrule(lr){5-7}
    \cmidrule(lr){8-10}
    & Mean & Median & Max.
    & Mean & Median & Max.
    & Mean & Median & Max. \\
    \midrule
    New
      & 0.003 & 0.003 & 0.016
      & 0.009 & 0.007 & 0.091
      & 0.018 & 0.008 & 13.413 \\
    qfautors
      & 13.754 & 0.004 & 14780.720
      & 0.011 & 0.007 & 0.473
      & 0.040 & 0.008 & 42.407 \\
    Magma
      & -- & 0.100 & --
      & 0.376 & 0.200 & 3.410
      & 0.540 & 0.260 & 69.110 \\
    \bottomrule
  \end{tabular}
\end{table}
On certain lopsided lattices in the dataset the new algorithm is about a million times faster than state-of-the-art methods while having a comparable median and a smaller maximum running time (cf. \Cref{tab:automorphism-group-timings}). Lattices of small determinant tend to be less lopsided. This is reflected in the datasets (B) and (C), which consist of unimodular lattices of ranks 25 and 26 without norm-$1$ vectors and are less lopsided than the ones in (A), which are of rank 17 and have determinant $90$, $112$ or $120$; see \Cref{fig:lopsidedness-distributions}.

\subsection*{Canonical forms}

Canonical forms are particularly useful for classifying large collections of
lattices using Kneser's neighbor method and for identifying lattices in databases such as the LMFDB~\cite{lmfdb}. Computing a canonical form for each
lattice reduces repeated isometry tests to comparisons of matrices. More
generally, any complete invariant can be used for this purpose. 

Dutour Sikiri\'c, Haensch, Voight and van~Woerden~\cite{SHVW} developed an
algorithm for computing canonical forms of positive definite lattices using
characteristic vector sets. Together with graph canonicalization this yields a complete
invariant, and the resulting ordering of the vectors can be used to recover
a canonical Gram matrix.

We refine this construction for lattices with roots
by incorporating the action of the Weyl group, which results in much smaller graphs. This gives a significant practical improvement: on our datasets,
using the Weyl group reduces the vector sets by mean factors of
up to $6.3\cdot 10^3$.
Equally important, the total canonical form computation takes time comparable to our automorphism
group computation; see \S{}\ref{subsec:inv-exp} and
\Cref{tab:characteristic-vector-set-data}.

\subsection*{Organization}
In \S{}\ref{sec:plesken-souvignier}
we recall the Plesken--Souvignier algorithm and discuss its limitations.
Section~\ref{sec:variations} develops variants for strongly well-rounded,
well-rounded and lopsided lattices, together with a method for compressing
the additional bilinear forms used in these computations.
Section~\ref{sec:weyl} reviews work of Chenevier and Ta\"ibi on how the Weyl group can be used to reduce
automorphism group computations. Next, \S{}\ref{sec:targeted-enumeration}
uses invariant orthogonal decompositions to improve the enumeration of
vectors required by the Plesken--Souvignier algorithm. Section~\ref{sec:isometry} treats isometry search. 
Section~\ref{sec:invariants}
incorporates the Weyl group action into the construction of complete
invariants and canonical forms via graph canonicalization.
Finally, \S{}\ref{sec:exp} presents our implementations and computational
experiments, comparing the performance of the new algorithms with existing
methods on several datasets of lattices.

\subsection*{AI disclosure}
We used AI for coding assistance. The mathematical ideas and the basic structure of the code are exclusively ours. An AI generated examples of strongly but not perfectly well-rounded lattices for \Cref{example:strongly-well-rounded} and suggested the reference for \Cref{lem:close_vector_root}. It also helped with final proofreading of the manuscript.

\subsection*{Acknowledgments}
We thank Stevell M\"uller for discussions and encouragement.
The second author also thanks Ga\"etan Chenevier and Olivier Ta\"ibi
for helpful discussions.
The authors gratefully acknowledge support from the Deutsche Forschungsgemeinschaft --- Project-ID 286237555 --- TRR 195.

\section{The algorithm of Plesken and Souvignier}\label{sec:plesken-souvignier}
The algorithm of Plesken and Souvignier \cite{plesken-souvignier:1997} takes as input
a list of integer matrices $G_1, \dots, G_n \in \ZZ^{r \times r}$ such that
$G_1$ is positive definite and symmetric.
It outputs generators for the group
\[\{U \in \GL_r(\ZZ) \mid \forall i \in \{1,\dots, n \}: U G_i U^t = G_i \}.\]
In more abstract terms:
Let $L$ be a finitely generated free abelian group of rank $r$ and
\[b=(b_1,\dots, b_n)\colon L \times L \longrightarrow \ZZ^n,\, (x, y) \longmapsto (b_1(x,y), \dotsc,b_n(x,y)) \]
a $\ZZ^n$-valued bilinear form such that its first component $b_1$ is symmetric and positive semidefinite.
The pair $(L,b)$ is called a \textit{vector-valued lattice} and
its automorphism group is defined to be 
\[\Aut(L,b)=\{ f \in \GL(L) \mid \forall x,y \in L: b(f(x),f(y))=b(x,y)\}.\]
For the remainder of this section, we assume that $b_1$ is positive definite.
The concrete setting of the Plesken--Souvignier algorithm can be recovered by choosing an ordered $\ZZ$-basis $(e_1,\dots, e_r)$ for $L$, that is, an isomorphism $L \to \ZZ^r$, and by considering $G_i := (b_i(e_k,e_l))_{1\leq k,l\leq r}$ for $1 \leq i \leq n$. An isomorphism of two lattices $(L,b)$ and $(L',b')$ is a $\ZZ$-linear bijection $f\colon L \to L'$ with $b'(f(x),f(y))=b(x,y)$ for all $x,y\in L$. A variant of the Plesken--Souvignier algorithm tests whether $L\cong L'$ and computes an isomorphism if it exists.
\begin{example}
For $n=1$ the pair $(L,b) = (L,b_1)$ is a classical positive definite quadratic integer lattice and $\Aut(L,b) = O(L,b_1)$ is its orthogonal group. This is the most important use case of the Plesken--Souvignier algorithm.
\end{example}

The option of specifying various forms $G_i$ as input gives the algorithm of Plesken--Souvignier a great deal of flexibility, which allows for important applications outside the classical setting of integer lattices. For instance, one can apply it to automorphism and isomorphism computations of quadratic or Hermitian lattices over number fields via the transfer construction (that is, restriction of scalars); see, for example,~\cite{kirschmerhabil}.
We will use this flexibility to speed up the computation of automorphism groups of positive definite integer lattices.

\subsection*{Overview} Before we continue, we give an overview of the Plesken--Souvignier algorithm (see \cite{plesken-souvignier:1997} for the details) and its limitations. For simplicity, we consider the case of automorphisms only.
We let $e_1,\dots, e_r$ be a fixed basis of $L$, which we use to identify $L$ with $\ZZ^r$.
For $0 \leq k \leq r$, a \textit{$k$-partial automorphism} is a $k$-tuple $(s_1,\dots, s_k) \in L^k$ such that $b(e_i,e_j) = b(s_i,s_j)$ for all $i, j=1,\dots, k$. Thus, a $k$-partial automorphism is just an isometric embedding
\[\langle e_1,\dots, e_k\rangle_\ZZ \lhook\joinrel\longrightarrow L, \, e_i \longmapsto s_i.\]
\begin{enumerate}
 \item Let $m=\max\{(G_1)_{ii} \mid i=1,\dots,r\}$ be the largest entry on the diagonal of $G_1$. Since $G_1$ is positive definite, the set $S_m=\{x \in L \mid b_1(x,x)\leq m\}$ is finite.%
 \item Given a $k$-partial automorphism $(s_1,\dots, s_k)$, it may or may not extend to
 a $(k+1)$-partial automorphism $(s_1,\dots, s_{k+1})$, but if it does, then $s_{k+1} \in S_m$. Thus there are only finitely many ways to extend a $k$-partial automorphism to a $(k+1)$-partial automorphism.
 \item The notion of extension turns the set of all partial automorphisms into a rooted directed tree $T$, whose vertices of level $k$ are the $k$-partial automorphisms.
 The leaves of $T$ at level $r$ correspond to the isometries of $(L, b)$.
 \item
 One can find an automorphism by performing a depth-first search in the tree $T$.
 The smaller the tree, the faster the search.
\end{enumerate}
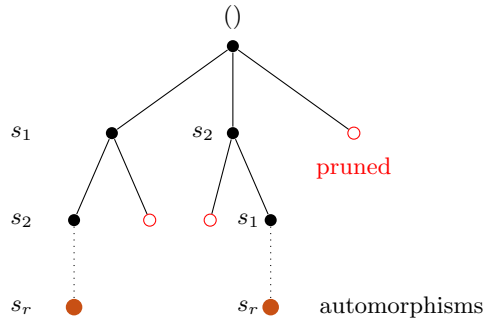
\begin{center}
     \begin{tikzpicture}[
          x=1cm,y=1.15cm,
          nd/.style={circle,fill=black,inner sep=1.6pt},
          dead/.style={circle,draw=red,inner sep=1.6pt},
          leaf/.style={circle,fill=accent,inner sep=2.2pt}]
        \node[nd,label=above:{\scriptsize $()$}] (r) at (0,0) {};
        \node[nd] (a) at (-1.6,-1) {};
        \node[nd] (b) at (0,-1) {};
        \node[dead] (c) at (1.6,-1) {};
        \draw (r)--(a) (r)--(b) (r)--(c);
        \node[nd] (a1) at (-2.1,-2) {};
        \node[dead] (a2) at (-1.1,-2) {};
        \node[dead] (b1) at (-0.3,-2) {};
        \node[nd] (b2) at (0.5,-2) {};
        \draw (a)--(a1) (a)--(a2) (b)--(b1) (b)--(b2);
        \node[leaf] (a11) at (-2.1,-3) {};
        \node[leaf] (b21) at (0.5,-3) {};
        \draw[dotted] (a1)--(a11) (b2)--(b21);
        \node[red,font=\scriptsize] at (1.6,-1.4) {pruned};
        \node[font=\scriptsize,anchor=west] at (1.0,-3) {automorphisms};
        \node[font=\scriptsize,anchor=east] at (-2.5,-1) {$s_1$};
        \node[font=\scriptsize,anchor=east] at (-2.5,-2) {$s_2$};
        \node[font=\scriptsize,anchor=east] at (-2.5,-3) {$s_r$};  \node[font=\scriptsize,anchor=east] at (-0.1,-1) {$s_2$};
        \node[font=\scriptsize,anchor=east] at (0.5,-2) {$s_1$};
        \node[font=\scriptsize,anchor=east] at (0.5,-3) {$s_r$};
      \end{tikzpicture}
      \end{center}
To make this approach feasible, it is crucial to prune branches of the tree that do not extend to the highest level $r$. That is, one must detect as early as possible when a $k$-partial automorphism cannot be extended to an $r$-partial automorphism. To this end, Plesken and Souvignier introduce various invariants and optimizations, such as the fingerprint, vector sums and Bacher polynomials.
After finding a subgroup $H \leq \Aut(L,b)$, one continues the computation in $T/H$ via orbit representatives and stabilizers.

\begin{remark}\label{remark:coloring}
It will turn out to be useful in \cref{sect:reduced-automorphisms} to extend the Plesken--Souvignier algorithm very slightly by allowing a coloring of the vectors as input, for instance, in terms of linear forms $L \to \ZZ$.
\end{remark}

\subsection*{Limitations} Let us now comment on the challenges faced by the algorithm.
Theoretically, Step (i) is the hardest problem, since it includes the computation of a shortest vector, which is known to be NP-hard under randomized reductions by Ajtai~\cite{Ajtai:1998}.
However, for most interesting lattices (of dimension $\leq 32$) coming from theoretical applications, shortest vectors can be efficiently enumerated using the algorithms of Fincke--Pohst \cite{FinckePohst1985} or Kannan~\cite{Kannan1983}; see also \cite{Micciancio2015}. In any case, it is worth using lattice reduction on $G_1$ so that $m$ is small.
If $m$ is not small enough, then $\lvert S_m \rvert$ will be too large to store in memory and the Plesken--Souvignier algorithm fails.

For a lattice $(L,(b_1,\dotsc,b_n))$, we define its minimum and its set of shortest vectors to be those of $(L,b_1)$.
The practical runtime of the Plesken--Souvignier algorithm depends heavily on the
properties and the distribution of the shortest vectors.
For example, observations indicate that the algorithm performs very well when the diagonal entries of $G_1$ are all equal to the squared length of a shortest vector, as is the case for the Leech lattice.

\begin{definition}
 A lattice $(L,b) = (L,(b_1,\dots, b_n))$ is called
 \begin{enumerate}
  \item perfectly well-rounded if $L$ has a $\ZZ$-basis of shortest vectors;
  \item strongly well-rounded if $L$ is generated by its shortest vectors;
  \item well-rounded if its shortest vectors span
$L \otimes_{\ZZ} \QQ$;
  \item lopsided if $L$ is not well-rounded.
 \end{enumerate}
\end{definition}
Given the set of shortest vectors of $L$, it is straightforward to decide whether $L$ is strongly well-rounded using linear algebra over $\ZZ$, and whether it is well-rounded or lopsided using linear algebra over $\QQ$.
In view of these notions, a lattice $(L, b)$ is less favorable for the classical Plesken--Souvignier algorithm if
\begin{itemize}
    \item
    $(L, b)$ is lopsided; or
    \item
    $(L, b)$ is well-rounded, but not perfectly well-rounded,
\end{itemize}
because in these cases every basis satisfies $m > \min(L)$, resulting in larger sets $S_m$.

\begin{example}
 The Leech lattice is the unique even unimodular $\ZZ$-lattice of rank $24$ and minimum $4$. It is perfectly well-rounded and has $196\,560$ shortest vectors.
 The $24$-dimensional unimodular lattice $E_8^3$ is also perfectly well-rounded.
 The remaining 22 Niemeier lattices are well-rounded but not strongly well-rounded.

 By \cite[Theorem~1.1]{MartinetSchuermann2012}, strongly well-rounded
 lattices of rank at most $9$ are perfectly well-rounded.
 A strongly well-rounded lattice of rank $10$ that is not perfectly
 well-rounded is constructed in \cite[Section~8]{MartinetSchuermann2012}.
\end{example}
For lopsided lattices, $m$ can be large enough that enumerating $S_m$ with the original algorithm of Plesken--Souvignier becomes impractical.
\begin{example}
  Let $m \in \Z_{\geq 3}$. The lattice $L = E_8 \oplus [m]$ has automorphism group $O(E_8)\times \{\pm \id_{[m]}\}$. Via the theta series of $E_8$, we see that the set $S_m$ is of cardinality 
\[ 3+240\sum_{i=1}^{\lfloor m/2 \rfloor}\sigma_3(i), \quad  \text{where } \sigma_3(n) = \sum_{d\mid n} d^3. \]
This cardinality grows like $m^4$. Already for $m=20$, one has $\lvert S_m \rvert = 794\,163$.
\end{example}
One way to avoid this enumeration is to decompose $L$ as $E_8 \oplus [m]$ and then compute $O(E_8)$ and $O([m])$ separately. However, such orthogonal decompositions are rare (as can be seen from the mass formula).
\begin{example}\label{ex:lop-indec}
 For $m\geq 3$, the lattice $L_m$ with Gram matrix
 \[\begin{pmatrix}
 2 & 1 & 0\\
 1 & 2 & 1\\
 0 & 1 & m
 \end{pmatrix}\]
is lopsided and indecomposable.
\end{example}

\section{Variants}\label{sec:variations}
In this section we present variants and improvements of the classical Plesken--Souvignier algorithm for strongly well-rounded, well-rounded, and lopsided lattices. Throughout, let $(L,b)$ be a vector-valued lattice with $b=(b_1,\ldots,b_n)$, where $b_1$ is symmetric and positive definite. As before, lengths, shortest vectors, and orthogonal complements are taken with respect to $b_1$.

\subsection{Strongly well-rounded}\label{subsec:strongly}

Assume now that $L$ is strongly well-rounded with minimum $m$.
Then $\Smin$ is a generating set of shortest vectors. If $L$ is not perfectly well-rounded, then $\Smin$ does not contain a basis. In particular, in the Plesken--Souvignier algorithm one has to select some basis and work with $S_k$ for some $k > m$. But the set $S_{k}$ may be much larger than $\Smin$. 

\begin{example}\label{example:strongly-well-rounded}
In \cite[Dataset E]{LatticeAutBench}, we give an example of an indecomposable strongly but not perfectly well-rounded lattice $L$ of rank $30$ and determinant $2^{12} 3^9  7^2  13^2$. The lattice $L$ is generated by its $62$ shortest vectors of squared length $m = 6$.
But the smallest $k$ for which $S_k$ contains a $\ZZ$-basis is $8$,
and $L$ has $3\,516$ vectors of squared length $8$.
\end{example}

We now present an algorithm for computing $\Aut(L, b)$, which requires backtracking only with the set $\Smin$.
To this end, let $S=\Smin$ and select a generating set $s_1,\dots, s_g$ of $L$ from $S$, which, for reasons of efficiency, should be minimal with respect to inclusion. 
Let $C = \ZZ^g$ with standard basis $(e_1,\dotsc,e_g)$ and let $\psi \colon C \to L$, $e_i \mapsto s_i$. 
We obtain an exact sequence
\[0 \longrightarrow K \longrightarrow  C \xrightarrow[]{\, \ \psi \ \,} L \longrightarrow 0\]
where $K = \ker(\psi)$. 
We define the form on $C$ as the pullback $\psi^*b$, that is, $\psi^*b(x, y) = b(\psi(x), \psi(y))$ for all $x, y \in C$.
Then $(C,\psi^*b)$ is a vector-valued lattice with $c_1$ positive semidefinite, and $\psi \colon C \to L$ preserves the forms.

Let $F$ be the set of $\mathbb Z$-linear maps $\psi' \colon C \to L$
that preserve the forms and satisfy $\psi'(e_i) \in S$ for
$1 \leq i \leq g$. Since such a map is uniquely determined by the
images of the basis vectors $e_i$, we can identify $F$ with the finite set
\[
  \{(f_1,\ldots,f_g) \in S^g
    \mid b(f_i,f_j)=b(s_i,s_j)
    \text{ for all } 1 \leq i,j \leq g\}.
\]
We can enumerate this set using a backtracking search as in the Plesken--Souvignier algorithm.
The following lemma shows that we can easily recover $\Aut(L, b)$ from $F$:

\begin{lemma}
Every $\psi' \in F$
satisfies $K \subseteq \ker(\psi')$ and induces a unique automorphism
$f \in \operatorname{Aut}(L,b)$ with $\psi'=f\circ\psi$.
In particular, the map
\[
  \Aut(L,b) \longrightarrow F,
  \, f \longmapsto f\circ\psi,
\]
is a bijection. Its inverse is given by
$\psi' \mapsto \psi'\circ\phi$, where
$\phi \colon L \to C$ is any $\mathbb Z$-linear section of $\psi$.
In particular, this inverse is independent of the choice of section.
\end{lemma}

\begin{proof}
Let $\psi' \in F$. By definition, $\psi'^*b=b$. For every $k \in K$, we have
\[
  b_1(\psi'(k),\psi'(k))
  = (\psi'^*b)_1(k,k)
  = b_1(\psi(k),\psi(k))
  = 0.
\]
Positive definiteness of $b_1$ implies that $\psi'(k)=0$.
Thus $\psi'$ factors uniquely through $C/K \cong L$, yielding
a $\mathbb Z$-linear endomorphism $f$ of $L$ with
$\psi'=f\circ\psi$. Since $\psi$ is surjective and both $\psi$
and $\psi'$ preserve the forms, $f$ preserves every component of $b$.
Since $f$ preserves the positive definite form $b_1$, it is bijective and thus belongs to $\Aut(L, b)$.

Conversely, every $f \in \Aut(L, b)$ preserves the set
$S$ of shortest vectors, so $f\circ\psi \in F$.
It follows that $f \mapsto f \circ \psi$ is a bijection.
Finally, a $\mathbb Z$-linear section $\phi$ exists
because $L$ is free, and
$
  \psi'\circ\phi
  = f\circ\psi\circ\phi
  = f
$.
This proves the formula for the inverse and its independence of
the chosen section.
\end{proof}

\begin{remark}
Experiments on strongly but not perfectly well-rounded lattices in \cite[Dataset E]{LatticeAutBench} yielded speedups by factors ranging from 2 to 25.
Since such lattices are rather scarce, we did not include this variant in the general-purpose implementation shipped with \textsc{OSCAR}. 
\end{remark}

\subsection{Well-rounded}
The idea of the strongly well-rounded case can also be applied to the general well-rounded case. For this, set 
\[m' = \min\{k \in \Z_{\geq 0} \mid \langle S_k \rangle = L\}, \]
and apply the construction of \S{}\ref{subsec:strongly} with $S = S_{m'} \setminus \{0\}$, choosing a generating set of $L$ from $S$.

\subsection{Lopsided}\label{subsec:lopsided}
We now consider improvements for lopsided lattices.
In order to speed up the computation of $\Aut(L, b)$ with the Plesken--Souvignier algorithm, we wish to find additional nontrivial bilinear forms $c_1,\dotsc,c_l$ on $L$ such that
\[ \Aut(L, b) = \Aut(L, (b_1,\dotsc,b_n,c_1,\dotsc,c_l)) \]  holds.
Set $G = \Aut(L, b)$.
Note that
the bilinear forms $c_1,\dotsc,c_l$ we are looking for must be $G$-invariant bilinear forms, and thus are intimately
linked to the structure of the $\Q[G]$-module $V = L \otimes \Q$  (since the space of $\Q$-valued $G$-invariant bilinear forms on $V$ is isomorphic to $\operatorname{Hom}_{\QQ[G]}(V, V^\ast)$, where $V^*$ is the dual of $V$).

More precisely, assume that we are given a decomposition
\[ V = V_1 \oplus \dotsb \oplus V_l \]
of $V$ into $\Q[G]$-submodules. For each $1 \leq i \leq l$, the projection $\pi_i \colon V \to V_i$ induces a $G$-invariant symmetric bilinear form
\[ c_i \colon L \times L \longrightarrow \Q, \, (x, y) \longmapsto b_1(\pi_i(x), \pi_i(y)), \]
which can be scaled to be integer-valued.
The $G$-invariance of the $c_1,\dotsc,c_l$ implies that $\Aut(L, b) = \Aut(L, (b_1,\dotsc,b_n, c_1,\dotsc,c_l))$.
Since the group $G$ is the object we are trying to compute, we do not know the structure of $V$ as a $\Q[G]$-module.
On the other hand, we will see that for lopsided lattices, we can find a nontrivial decomposition of $V$ using the shortest vectors of $L$, without knowledge of $G$.

\begin{definition}
Let $L$ be a lattice. 
Set $M_1=L$. For each $i$ with $M_i \neq 0$, let $L_i=\langle \Smin(M_i)\rangle_{\Z}$ and $M_{i+1}=L_i^\perp\cap M_i$. Let $l$ be such that $M_{l+1}=0$.
The lattices $L_1,\dotsc,L_l$ are called the \emph{successive sublattices} of $L$.
We denote by 
\[\pi_i\colon L \to L_i \otimes \QQ\]
the corresponding orthogonal projections.
\end{definition}

\begin{lemma}
  Let $L_1,\dotsc,L_l$ be the successive sublattices of $L$ and let $G = \Aut(L, b)$. Then $L_i \otimes \Q$, $1 \leq i \leq l$, is $G$-invariant and $L_1 \oplus \dotsb \oplus L_l \subseteq L$ is a finite-index sublattice.
\end{lemma}

\begin{proof}
  The claim follows by induction and the following two observations. First, the set of shortest vectors of a $G$-invariant lattice is $G$-invariant, and thus spans a $G$-invariant lattice. Second, orthogonal complements and intersections of $G$-invariant lattices are again $G$-invariant.
  If $M_i\ne0$, then $L_i\ne0$. Since $b_1$ is positive definite,
\[
M_i\otimes\mathbb Q
=(L_i\otimes\mathbb Q)\oplus(M_{i+1}\otimes\mathbb Q).
\]
Thus the ranks of the $M_i$ strictly decrease until the construction terminates. In this case, $L_1\oplus\cdots\oplus L_l$ has the same rank as $L$, and hence finite index in $L$.
\end{proof}

\begin{example}
Consider the lattice $L=L_m$ from \Cref{ex:lop-indec} with basis $e_1,e_2,e_3$ and Gram
matrix
\[
 \begin{pmatrix}
  2 & 1 & 0\\
  1 & 2 & 1\\
  0 & 1 & m
 \end{pmatrix},
 \qquad m\geq 3.
\]
Its shortest vectors have squared length $2$ and are given by
\[
 \Smin(L)
 =
 \{\pm e_1,\pm e_2,\pm(e_1-e_2)\}.
\]
Hence its first successive sublattice is
 $L_1=\langle e_1,e_2\rangle_{\ZZ}\cong A_2$.
A direct calculation gives
\[
 M_2=L_1^\perp\cap L
 =\ZZ(e_1-2e_2+3e_3).
\]
Since $M_2$ has rank one, its shortest vectors generate it, and
therefore
 $L_2=M_2$ and
 $M_3=0$.
Moreover, $[L:L_1\oplus L_2]=3$.
Thus, the successive sublattices yield a nontrivial orthogonal
decomposition of $L\otimes\QQ$, although their direct sum is a proper
sublattice of $L$.
\end{example}

Suppose now that $L$ is lopsided. Its shortest vectors do not span
$L\otimes\QQ$, so $M_2\neq 0$ and $L$ has at least two successive
sublattices. The associated projections therefore yield nontrivial
$\Aut(L,b)$-invariant bilinear forms $c_1,\dotsc,c_l$. Adding these
forms does not change the automorphism group:
\[
\Aut(L,b)
=
\Aut(L,(b_1,\dotsc,b_n,c_1,\dotsc,c_l)).
\]
In the Plesken--Souvignier search, the additional forms impose further
compatibility conditions when extending partial automorphisms and can
thereby reduce the search tree $T$.

\subsection{Compression of bilinear forms on finite sets}

The running time of the Plesken--Souvignier algorithm is sensitive to the number of bilinear forms supplied as input.
The following lemma shows how to replace two positive semidefinite forms by a single one preserving all scalar-product information on a prescribed finite set of vectors.
By applying this construction repeatedly, we can compress any finite collection of such forms into one.
After computing the set $S$, we apply this compression to reduce the number of forms tested during backtracking search, which can improve its running time.

\begin{lemma}\label{lem:compression}
Let $b_1,b_2$ be symmetric, positive semidefinite, integer-valued
bilinear forms on $L$, and let $\emptyset\ne S\subseteq L$ be finite.
Set
\[
  M=\max_{s\in S}b_2(s,s).
\]
If $\lambda\in\ZZ_{>0}$ satisfies $\lambda>2M$, then
$b=\lambda b_1+b_2$ has the following property: for all
$x,y,x',y'\in S$,
\[
  b(x,y)=b(x',y')
  \quad\Longrightarrow\quad
  b_i(x,y)=b_i(x',y')
  \quad\text{for }i=1,2.
\]
In particular, if $S$ is an $\Aut(L, (b_1, b_2))$-invariant generating set of $L$, then
\[
  \Aut(L, (b_1, b_2)) = \{g \in \Aut(L, b) \mid g(S) = S\}.
\]
\end{lemma}

\begin{proof}
By Cauchy--Schwarz, we have
\[
  \lvert b_2(x,y) \rvert
  \leq\sqrt{b_2(x,x)b_2(y,y)}
  \leq M
\]
for all $x, y \in S$.
Suppose that $b(x,y)=b(x',y')$ for $x,y,x',y' \in S$. Then
\[
  \lambda\bigl(b_1(x,y)-b_1(x',y')\bigr)
  = b_2(x',y')-b_2(x,y).
\]
The absolute value of the right-hand side is at most $2M<\lambda$.
Since $b_1(x,y)-b_1(x',y')$ is an integer, it must vanish.
The equality above then also gives $b_2(x,y)=b_2(x',y')$.

For the last assertion, every element of $\Aut(L, (b_1, b_2))$ preserves $b$ and, by assumption, $S$.
Conversely, if $g \in \Aut(L, b)$ satisfies $g(S) = S$, then the first assertion applied to $x,y,g(x),g(y) \in S$ shows that $g$ preserves both $b_1$ and $b_2$ on $S \times S$.
Since $S$ generates $L$, bilinearity implies that $g$ preserves both forms on $L$.
\end{proof}

Note that computing $M$ requires only $|S|$ evaluations of $b_2$, which are already performed during the computation of $S$ itself.

\section{Using the Weyl group}\label{sec:weyl}

In this section, we give an account of the use of root systems and their Weyl groups in the computation of lattice automorphism groups.
This idea was first described by Chenevier in~\cite{Chenevier2020, Chenevier2025} and was subsequently improved by Ta\"ibi~\cite{Taibi2024}, resulting in an efficient algorithm for automorphism group computations of lattices with roots, implemented in PARI/GP~\cite{PARI2}; see also~\cite[2.17]{ChenevierTabi2026}.
We will also recall known results on certain closest vector problems for root lattices, as they will play an important role in lattice invariants (see \S{}\ref{sec:invariants}).
In this section we will consider only classical integer lattices $(L, b)$ with a single symmetric, positive definite bilinear form $b$ on $L$.

\subsection{Preliminaries and notation for root lattices}\label{subsec:weyl-prelim}
See \cite{Ebeling2013, Bourbaki} for proofs and background. In the following, $(L,b)$ is an integral lattice.
A vector $\delta \in L$ with $b(\delta,\delta)=k\neq 0$ is called a \textit{$k$-root} if the reflection 
$s_\delta(x) = x - \frac{2b(\delta,x)}{b(\delta,\delta)}\,\delta$ in the hyperplane $\delta^\perp$ preserves $L$. Then $s_\delta \in \Aut(L, b)$. For simplicity, we only work with $2$-roots and simply call them roots. 

An integral lattice spanned by roots is called a \textit{root lattice}. A root lattice $L$ is called reducible if it is the direct sum of two root lattices $L = L_1 \oplus L_2$ of positive rank. Otherwise it is called irreducible. Irreducible root lattices are classified by their Coxeter--Dynkin diagrams, which are ADE diagrams.

More precisely, denote by $\Phi:=\Phi(L)$ the set of roots of $L$. Let $h \in L$ be such that $h^\perp \cap \Phi = \emptyset$. Then 
\[\Phi^+:=\{\delta \in \Phi \mid b(\delta,h)>0\}\]
is a \textit{positive root system}. The elements of $\Phi^+$ are called \textit{positive} and
$\Phi = \Phi^+ \cup -\Phi^+$. 
A positive root is called \textit{fundamental} (or \textit{simple}) if it is not the sum of two positive roots. The set of fundamental roots will be denoted by $\Delta$.
The fundamental roots $e_1,\dots, e_t$ form a basis of $\langle \Phi \rangle$. 
Its Coxeter--Dynkin diagram is a finite graph whose vertices are $e_1,\dots, e_t$ and where two vertices $e_i,e_j$ are connected by an edge if $b(e_i,e_j)=-1$. If $\langle \Phi \rangle$ is irreducible, then the corresponding Coxeter--Dynkin diagram $\Gamma$ is an $A_i$, $D_j$ or $E_k$ diagram, where $i\geq 1$, $j\geq 4$ and $k=6,7,8$, as in \cref{fig:ADE}.

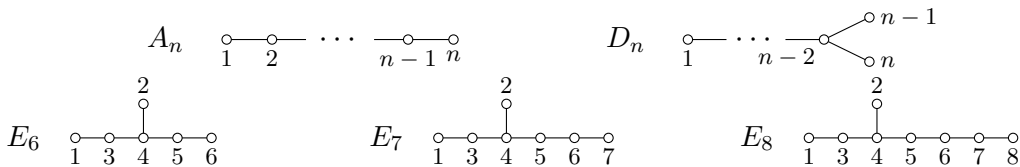
\begin{figure}[htbp]
\centering
\begin{tikzpicture}[
  x=4.5mm, y=4.5mm, line width=0.4pt,
  root/.style={circle, draw, fill=white, inner sep=0pt, minimum size=3.5pt},
  type/.style={anchor=east, font=\small},
  every label/.style={font=\scriptsize, inner sep=1pt},
  label distance=1pt
]
  \begin{scope}[xshift=2cm, x=6mm]
    \node[type] at (-0.7,0) {$A_n$};
    \draw (0,0) -- (1.75,0) (3.25,0) -- (5,0);
    \node at (2.5,0) {$\cdots$};
    \foreach \pos/\lab in {0/1,1/2,4/{n-1},5/n}
      \node[root, label=below:{$\lab$}] at (\pos,0) {};
  \end{scope}
  \begin{scope}[xshift=8.1cm, x=6mm]
    \node[type] at (-0.7,0) {$D_n$};
    \draw (0,0) -- (0.85,0) (2.15,0) -- (3,0) -- (4,0.65)
      (3,0) -- (4,-0.65);
    \node at (1.5,0) {$\cdots$};
    \node[root, label=below:{$1$}] at (0,0) {};
    \node[root, label=below left:{$n-2$}] at (3,0) {};
    \node[root, label=right:{$n-1$}] at (4,0.65) {};
    \node[root, label=right:{$n$}] at (4,-0.65) {};
  \end{scope}
  \foreach \rank/\offset in {6/0,7/4.8,8/9.7} {
    \begin{scope}[xshift=\offset cm, yshift=-1.3cm]
      \node[type] at (-0.7,0) {$E_{\rank}$};
      \draw (0,0) -- ({\rank-2},0) (2,0) -- (2,1);
      \node[root, label=below:{$1$}] at (0,0) {};
      \node[root, label=above:{$2$}] at (2,1) {};
      \foreach \index in {3,...,\rank}
        \node[root, label=below:{$\index$}] at ({\index-2},0) {};
    \end{scope}
  }
\end{tikzpicture}
\caption{The ADE diagrams. Vertex $i$ represents the fundamental root $e_i$.}
\label{fig:ADE}
\end{figure}

We define the \emph{Weyl} group of $L$ as the group 
\[W(L) := \langle s_\delta \mid \delta \in L \mbox{ is a root } \rangle = \langle s_\delta \mid \delta \in \Delta \rangle \leq \Aut(L, b) \]
generated by reflections in roots of $L$. The fundamental roots suffice to generate it. 

The closures of the connected components of $(L \otimes \RR) \setminus \bigcup_{\delta \in \Phi} \delta^\perp$ are called \emph{Weyl chambers}. The chamber $P$ containing $h$ is called the fundamental Weyl chamber. It is a fundamental domain for the action of $W(L)$ on $L \otimes \RR$. We have $P=\bigcap_{i=1}^t H_{e_i}$,
where $H_e:=\{x \in L \otimes \RR \mid b(x,e)\geq 0\}$ is the positive half-space defined by $e\in L\otimes \RR$. The walls of $P$ are given by the fundamental roots as follows: $P \cap e_1^\perp,\dots, P\cap e_t^\perp$. The \textit{Weyl vector} $\rho \in \langle \Phi \rangle_\QQ \subseteq L \otimes \QQ$ associated to a positive root system is defined as
\[\rho = e_1^\vee +\dots + e_t^\vee\] 
where $e_1^\vee,\dots, e_t^\vee$ form the basis of $\langle \Phi \rangle_\QQ$ dual to $e_1,\dots,e_t$ with respect to $b$, that is, with $b(e_i,e_j^\vee)=1$ if $i=j$ and $b(e_i,e_j^\vee)=0$ otherwise, for $1 \leq i, j \leq t$. It can be computed directly from a positive root system via the equation
\begin{equation}\label{eqn:rho}
    \rho = \frac 1 2 \sum_{\delta \in \Phi^+} \delta.
\end{equation}
To get a positive root system, fix any basis, compute $\Phi$ via a short vector enumeration, and declare $\delta \in \Phi$ to be positive, that is, in $\Phi^+$, if its first nonzero coordinate is positive.

\subsection{Close vectors in root lattices}\label{subsec:close-vectors}
The following lemma will be needed for the construction of characteristic vector sets in \S{}\ref{sec:invariants}. The result is standard. For the reader's convenience, we give a proof.
We use the setup from \S{}\ref{subsec:weyl-prelim}.

Let $R$ be a root lattice with roots $\Phi$. Following \cite[Ch.~VI, \S{}1, Exerc.~24]{Bourbaki}, we call a vector $y \in R^\vee$ \emph{minuscule} if $b(y,\delta) \in \{-1,0,1\}$ for all $\delta \in \Phi$.
If $R$ is irreducible with fundamental roots $e_1,\dots,e_t$ and highest root $\theta = \sum_{i} n_i e_i$, then the nonzero minuscule vectors in the fundamental Weyl chamber are exactly the $e_i^\vee$ with $n_i = 1$ (loc.\ cit.).

\begin{lemma}\label{lem:close_vector_root}\cite[Ch.~VI, \S{}1, Exerc.~24]{Bourbaki}
Let $R$ be a root lattice with positive root system $\Phi^+$ and fundamental roots $e_1,\dots, e_t$. Let $P$ be the corresponding fundamental Weyl chamber and denote by $e_1^\vee,\dots, e_t^\vee$ the dual basis defined by $b(e_i^\vee,e_i) = 1$ and $b(e_i^\vee,e_j) = 0$ for $i \neq j$.
Let $x \in R^\vee \setminus R$. Then there is a unique shortest vector in $(x + R) \cap P$. It is a sum of minuscule vectors, at most one per component.
\end{lemma}
\begin{proof}
We may assume that $R$ is irreducible. 
Let $y$ be a vector of shortest length in $(x + R) \cap P$. Let $\delta$ be any positive root. Suppose that $b(y,\delta) \geq 2$. Then $y-\delta$ is strictly shorter than $y$ and we can use an element of the Weyl group to transport it to $P$. Since $W(R)$ acts trivially on the discriminant group $R^\vee/R$, the image of $y-\delta$ is actually in $(x+R)\cap P$, which contradicts the minimality of $y$. Hence $b(y, \delta) \in \{0,1\}$ for all positive roots $\delta$, that is, $y$ is minuscule.
By \cite[Lemma 1.18]{Ebeling2013} and \cite[Ch.~VI, \S{}2, Exerc.~5\,a)]{Bourbaki}, the minuscule $e_i^\vee$ form a system of representatives of the nonzero cosets of $R^\vee/R$. Hence $y$ is the unique minuscule $e_i^\vee$ in $x + R$.
\end{proof}

\subsection{The reduced automorphism group}\label{sect:reduced-automorphisms}
Let $L$ now be any lattice.
We now explain Chenevier's improvements for the automorphism group computation.
The Weyl group acts simply transitively on the set of Weyl chambers and decomposes $\Aut(L, b)$ as a semidirect product
\[\Aut(L, b) = W(L) \rtimes \Aut(L,\rho)\]
where 
$\Aut(L,\rho) = \{g \in \Aut(L, b) \mid g(\rho)= \rho\}$.
We call the quotient
\[\Autred(L) = O(L)/W(L) \cong \Aut(L,\rho)\]
the \emph{reduced automorphism group} of $L$ and identify it with $\Aut(L,\rho)$ when $\rho$ is fixed.  
The decomposition $\Aut(L, b) = W(L) \rtimes \Aut(L,\rho)$ reduces the problem to finding generators for $W(L)$ and $\Aut(L, \rho)$ separately.
For the Weyl group $W(L)$, generators are readily available after computing a fundamental root system.
Generators for the typically much smaller group $\Aut(L,\rho)$ can be determined with the usual automorphism group algorithm, using the linear form $x \mapsto b(\rho, x)$ as an additional invariant
(cf. \Cref{remark:coloring}).

\subsection{Invariant subspaces for the reduced automorphism group}
Since the reduced automorphism group $\Aut(L,\rho)$ is (in general) smaller than $\Aut(L)$,
restricting from $\Aut(L)$ to $\Aut(L,\rho)$ may yield a finer decomposition into invariant subspaces and therefore additional invariant forms.

To find them, let $R$ be the root sublattice of $L$ and $\rho$ a Weyl vector.
Then $\Aut(L,\rho)$ acts as a permutation of the fundamental roots which preserves the (reducible) Coxeter--Dynkin diagram $\Gamma$ of $R$. Therefore, it also permutes the connected components of any given type $A_i$, $D_j$ or $E_k$. We call the union of all components of a given type $\tau$ the \emph{$\tau$-part} of $\Gamma$. Further, $\Aut(L,\rho)$ can act on individual components as follows:
the components of type $A_i$, $i>1$, could be ``flipped'', the $3$ leaves of $D_4$ could be permuted, as well as two leaves of $D_j$, $j>4$, and $E_6$ also has a ``flip''. For $L=R$ we obtain in total
\begin{enumerate}
\item $\lceil i/2 \rceil$ orbits for each $A_i$-part;
\item $2$ orbits for each $D_4$-part;
\item $j-1$ orbits for each $D_j$-part with $j>4$;
\item $4$, $7$ and $8$ orbits for each $E_6$-, $E_7$- or $E_8$-part, respectively. \end{enumerate}
Following Ta\"ibi \cite{Taibi2024}, we use the orbit sums, which are invariant vectors and hence yield invariant linear forms.

As an $\Aut(L,\rho)$-module, $R\otimes\QQ$ is a direct sum of permutation modules, one for each orbit. The isotypic components of $R\otimes\QQ$ yield further invariant bilinear forms, which can then be compressed with \Cref{lem:compression} and used in \Cref{sec:targeted-enumeration}.

For $L=R$ this decomposition is determined by the Coxeter--Dynkin diagram $\Gamma$ alone. In general, however, the natural map
\[\Aut(L,\rho) \longrightarrow \Aut(R,\rho), \ g \longmapsto g|_R\]
is not surjective, and its image can have a finer isotypic decomposition, because the embedding in $L$ can break the symmetry of the $\tau$-parts.
We just give a simple example here: 
Let $R_1$ and $R_2$ be two irreducible root sublattices of $L$ of the same type and $\pi_i \colon L \to R_i \otimes \QQ$ the corresponding orthogonal projections.
If the abelian groups $\pi_1(L)/R_1$ and $\pi_2(L)/R_2$ are not isomorphic, then $R_1$ and $R_2$ cannot be permuted. So there is another invariant vector and the orthogonal subspaces split further. This gives additional invariant bilinear forms.

\section{Targeted enumeration}\label{sec:targeted-enumeration}

One of the main computational tasks of the Plesken--Souvignier algorithm is the initial computation of short vectors. 
Let $(L,b)$ be a vector-valued lattice with $b_1$ positive definite and $V = L \otimes \QQ$ its rational span.
Furthermore, let $H \leq \Aut(L,b)$ be a subgroup. For a nonzero element $e$ of $L$, one needs to determine a finite $H$-invariant set $S \subseteq L$ containing $e$ (without explicit knowledge of $H$).
The original algorithm of Plesken--Souvignier uses 
$S(e) = \{x \in L \mid b_1(x, x) = b_1(e, e)\}$ and takes 
$S = \bigcup_{i=1}^r S(e_i)$ where $e_1,\dots e_r$ form a basis of $L$.
The sets $S(e)$ can be determined using lattice enumeration. But this can become a bottleneck when $b(e, e)$ or the rank of $L$ is large. 

We now describe a different choice of $S$, which exploits an orthogonal sum decomposition of $V$, given for example by the successive sublattices of $L$ (see \S{}\ref{subsec:lopsided}).
More precisely, assume that $V$ admits a $b_1$-orthogonal decomposition $V = \bigoplus_{i=1}^l V_i$ with $H$-invariant subspaces $V_i$ and such that $V_1$ is fixed pointwise by $H$. We allow $V_1=0$. 
Denote by $\pi_i \colon V \to V_i \hookrightarrow V$ the corresponding projections. We set $c_i(x,y) = b(\pi_i(x),\pi_i(y))$ and $c = (c_1,\dots, c_l)\colon V \times V \to (\QQ^n)^l$. Note that the $c_i$ are vector-valued bilinear forms on $V$ and that each first component $(c_i)_1$ is positive definite when restricted to $V_i$.
By assumption, the bilinear forms $c_i$ are $H$-invariant and $\pi_1\circ g = \pi_1$ for all $g \in H$. Hence for any vector $e \in L$, the set 
\[S(e) =\{x \in L \mid c(x, x) = c(e, e) \text{ and } \pi_1(x)=\pi_1(e)\}\] 
is finite, $H$-invariant and contains $e$. We now describe how to compute it.

To this end, for $k=1,\dots, l$, we set $\Pi_k = \sum_{i=1}^k \pi_i$ to be the orthogonal projection onto $V_1\oplus \dots \oplus V_k$, $L_{(k)} = \Pi_k(L)$ and $\hat L_k = \pi_k(L)$.
Furthermore, we define $S_1(e)=\{\pi_1(e)\}$,
\[ S_k(e) = \{x \in L_{(k)} \mid \pi_1(x)=\pi_1(e) \text{ and }  c_i(x, x) = c_i(e, e) \text{ for all $i = 2,\dotsc,k$}\}. \]
Note that $L_{(l)} = L$ and that $S_l(e)=S(e)$.

The sets $S_1(e),\dotsc,S_l(e)$ can be determined recursively as follows.
Assume that $S_k(e)$ has been computed.
Since the bilinear form $c_k$ is positive definite on $V_k$, the set 
\[T_k(e) = \{t \in \hat{L}_k \mid c_k(t,t) = c_k(e,e)\}\] 
is finite and can be enumerated with the algorithm of Fincke--Pohst~\cite{FinckePohst1985}. 
Thus
\[S_{k+1}(e) = \{s+t \mid s \in S_k(e),\, t\in T_{k+1}(e),\, s+t \in L_{(k+1)}\}.\]
Note that $s+t$ is a priori only in $L_{(k)}\oplus \hat L_{k + 1}$ and so the containment in $L_{(k+1)}$ is the actual condition here. 

\begin{remark}
Turning this approach into an efficient alternative to a plain lattice enumeration requires various improvements, which we now sketch:

\begin{enumerate}
    \item Compute LLL-reduced bases of $\hat{L}_i$, $1 \leq i \leq l$, and combine them into a block-diagonal basis respecting the direct sum $\hat L_1 \oplus \dotsb \oplus \hat L_l$. Since $L \subseteq \hat L_1\oplus \dotsb \oplus \hat L_l$, we can represent all vectors with respect to this basis. 
    \item As part of the preprocessing, for $k=2,\dots, l$, we compute integral matrices $D_k$ and $U_k$ such that $x \mapsto x D_kU_k^{-1}$ represents the inclusion $\phi_{k} \colon L_{(k)} \to \hat{L}_1 \oplus \dots \oplus \hat{L}_{k}$, where the matrix $D_k$ is diagonal with entries $d_{k,i+1} \mid d_{k,i}$ and $U_k \in \GL_{n_k}(\ZZ)$ is invertible, where $n_k = \sum_{i=1}^k \dim_\Q(V_i)$. Then checking whether $x:=s+t$ is contained in $L_{(k+1)}$ reduces to one matrix-vector multiplication to determine $y = x \cdot U_{k+1}$, followed by testing whether $y_i$ is divisible by $d_{k+1,i}$ for all $i$. We can moreover delete the $j$th column of $U_k$ if $d_{k,j}=1$.   
    \item Naively testing all $\lvert S_k \rvert\cdot \lvert T_{k+1} \rvert$ pairs $(s,t)$ for containment in $L_{(k+1)}$ is expensive.
    To avoid this exhaustive matching, let
    \[\psi_k\colon \hat{L}_1 \oplus \dots \oplus \hat{L}_{k+1} \to \left(\hat{L}_1 \oplus \dots \oplus \hat{L}_{k+1}\right)/L_{(k+1)}\]
    be the quotient map and label the elements $s\in S_{k}(e)$ by $\psi_{k}(s)$, that is, we partition $S_k(e)$ into fibers with respect to $\psi_k$. For a given $t \in T_{k+1}(e)$, one then just looks for all $s \in S_{k}(e)$ with $\psi_{k}(s) = - \psi_{k}(t)$, that is, all elements of $\psi_{k}^{-1}(\{-\psi_k(t)\}) \cap S_k(e)$.
    Then one needs only $\lvert S_k(e) \rvert + \lvert T_{k+1}(e) \rvert$ evaluations of $\psi_k$.
    \item Since $t \in T_k(e)$ implies $-t \in T_k(e)$, it suffices to enumerate $T_k(e)$ only up to sign. This also applies to $S_{k}(e)$, if $\pi_1(e)=0$.
    \item In practice we want to compute $S(e_1), \dots, S(e_r)$ for several vectors $e_1,\dots, e_r \in L$ (for example, a basis of $L$). In this case the preprocessing can be shared, but also other parts of the computation such as the enumeration of the $T_k(e_i)$ and the computation of further invariant forms can be combined.
    \item Since for $k = 1,\dotsc,l$ the set $S_{k}(e)$ is $H$-invariant, its span $W_{k}(e)=\langle S_{k}(e)\rangle$ is $H$-invariant, and the orthogonal projection onto $W_k(e)$ commutes with $H$. Moreover, the vector $\alpha_{k}=\sum_{s \in S_{k}(e)} \pi_k(s)$ is fixed by $H$. These give an invariant bilinear form $w_k$ defined via the orthogonal projection and an invariant linear form $a_k(x) = b(\alpha_{k},x)$, $x \in L$. Let $e' \in L$ be another vector (for example, $e'=e_i$). The previous invariants can be used at each stage to replace $S_{k}(e')$ by the subset
    \[ \widetilde{S}_{k}(e')=\{s \in S_{k}(e_i) \mid w_k(s,s) = w_k(e',e') \text{ and } a_k(e') = a_k(s)\}. \]
    Then $\tilde{S}_l(e')$ is still $H$-invariant and contains $e'$. Note that $\alpha_k$ can only be nonzero if $-\id_L \notin H$, for example, for $H=\Aut(L,\rho)$ as in \Cref{sec:weyl}.
    \item It is more efficient to compute and work with $\pi_j(W_{k})\subseteq V_j$ for $j=2,\dots, k$, because linear algebra has cubic complexity in the dimension and $V_j$ is of smaller dimension than $V$. Note that $\pi_1(W_k)$ provides no additional information, because $V_1$ is fixed pointwise by $H$ in any case. 
\end{enumerate}
\end{remark}

\begin{example}
  We illustrate the idea with a small example. Consider the lattice $L$ with Gram matrix
  \[ G = \begin{pmatrix}
       2 &  1 &  0 &   0 \\
       1 &  8 & -2 &   2 \\
       0 & -2 &  8 &   0 \\
       0 &  2 &  0 & 192
     \end{pmatrix}, \]
  which is LLL-reduced.
  The lattice has only the two roots $\pm e_1$, and we choose $\rho = e_1/2$ as the Weyl vector.
  For $e = e_4$ and $H = \Aut(L, \rho)$ we wish to determine an $H$-invariant set $S \subseteq L$ containing $e$.
  The Plesken--Souvignier algorithm would determine the initial candidate set
  \[ \{ x \in L \mid b(x, x) = b(e, e)\} \]
  which in this case has cardinality $58$.

  We now compare this with our targeted enumeration.
  Since $\rho$ is a Weyl vector, the space $V_1 = \Q \cdot \rho$ is pointwise fixed by $H$.
  We choose $l = 2$ and $V_2 = V_1^\perp$.
  In particular, we have
  \[ \pi_1 \colon V \longrightarrow V, \ x \longmapsto (2x_1 + x_2)\cdot \rho \]
  and $\pi_2 = \id_V - \pi_1$.
  The semidefinite bilinear forms $c_1$ and $c_2$ induced by $\pi_1$ and $\pi_2$ have Gram matrices
  \[ G_1 = \begin{pmatrix}
       2 & 1   & 0 & 0 \\
       1 & 1/2 & 0 & 0 \\
       0 & 0   & 0 & 0 \\
       0 & 0   & 0 & 0
     \end{pmatrix}, \qquad G_2 = G - G_1. \]
  The set
  \[ S_2(e) = \{ x \in L \mid \pi_1(x) = \pi_1(e) \text{ and } c_2(x, x) = c_2(e, e) \}\]
  is exactly $\{\pm e_4\}$ and hence has only two elements.
\end{example}

\section{Isometry search}\label{sec:isometry}
The variant of the Plesken--Souvignier algorithm searching for an isometry $f\colon L \to L'$ between two lattices also benefits from the ideas of the previous sections, provided that the additional data on both sides are constructed so that every isometry matches them. This is the case for the successive sublattices of \S{}\ref{subsec:lopsided} and the induced forms $c_i$.
For the compression of \Cref{lem:compression}, one uses the same $\lambda$ on both sides. For the Weyl group reduction, any choice of Weyl vectors $\rho,\rho'$ works. The fixed spaces $V_1$ are typically built from orbit sums of fundamental roots. These orbits (and any other invariants) have to be identified in such a way that any isometry with $f(\rho)=\rho'$ matches them. Then $V_1$ can be used for targeted enumeration as in \S{}\ref{sec:targeted-enumeration} and the extra forms as input to the backtracking search. If the invariants do not match, for example, if the numbers of successive sublattices differ, then $L$ and $L'$ cannot be isometric.

\section{Complete invariants and canonical forms of lattices}\label{sec:invariants}

In this section we consider the problem of computing complete invariants of lattices. Such a complete invariant can be thought of as a computable function $i$ on the set of all positive definite lattices such that
\[
  i(L) = i(L') \quad\Longleftrightarrow\quad L \cong L'
\]
for all lattices $L$ and $L'$. (A more precise definition is given below.)
Common applications include efficiently partitioning collections of lattices into isomorphism classes and looking up lattices up to isomorphism.

A closely related problem is that of a ``canonical form'' of a lattice, which is the special case where $i(L)$ itself is a lattice (represented by its Gram matrix) with $i(L) \cong L$. An algorithm for computing a canonical form has been developed by Dutour Sikiri\'c, Haensch, Voight and van~Woerden~\cite{SHVW}.
Another variation of the problem is obtained by restricting the set of lattices under consideration.
An example is the Bacher--Venkov invariant introduced by Allombert
and Chenevier in~\cite{MR4950160}, which is efficiently computable and a posteriori turned out to be a complete invariant for unimodular lattices of rank $29$ having no
nonzero vectors of squared length at most $2$.
Both approaches first construct a natural finite set of vectors and then encode it as a graph.

Here, we will view invariants from a more abstract point of view,
which will allow us to naturally incorporate the action of the Weyl group.
The result will be a refined version of the invariant of~\cite{SHVW},
which performs well in practice; see \S{}\ref{subsec:inv-exp}.

\subsection{Invariants and functors}

It is natural to express the notion of invariants using the language of categories.
To this end, denote by $\LAT$ the category whose objects are the positive definite lattices and whose morphisms are the isometries.
Then a lattice invariant is a functor $F \colon \LAT \to \NN$, where $\NN$ is regarded as a discrete category. %
An invariant $F$ is complete if and only if the functor $F$ is essentially injective, that is, if for all lattices $L$ and $L'$, we have that $F(L) \cong F(L')$ implies $L \cong L'$.

The analogous problem for graphs is addressed by graph canonicalization algorithms implemented in packages such as \textsc{nauty}~\cite{MR3131381} or~\textsc{bliss}~\cite{bliss1,bliss2}.
Let $\GRAPH$ be the category whose objects are finite undirected graphs with loops and edges labeled by finite-length sequences of rational numbers $\QQ^{[\NN]}$ and whose morphisms are isomorphisms.
Efficient graph canonicalization provides us with a complete invariant $\GRAPH \to \NN$.
The construction of a complete invariant for lattices can therefore be reduced to the construction of an essentially injective functor $F \colon \LAT \to \GRAPH$.
Since larger graphs make the invariant more expensive to compute, our aim is to find functors $F$ whose images are small graphs.

\begin{example}\label{ex:cd-functor}
For a lattice $L$ denote by $\Phi(L)$ the set of its roots. 
Then $\Phi\colon \LAT \to \SET$
is a functor. 
Let $\Gamma(L)$ denote the Coxeter--Dynkin diagram of the root sublattice of $L$ labeled in a canonical fashion. This yields a functor $\Gamma \colon \LAT \to \GRAPH$ which maps all morphisms to the identity.
This gives an invariant which is not complete, since $\Gamma$ is not essentially injective. There are non-isomorphic lattices with isomorphic root sublattices. 
The construction depends somewhat awkwardly on a canonical labeling of the $ADE$-graphs and sends maps to the identity. This will be remedied later by defining the category of Weyl-marked lattices $\WLAT$, which also paves the way for a generalization. 
\end{example}

We now consider another construction of invariants of lattices.
Let $L$ be a lattice and let $X \subseteq L$ be a finite subset.
We view $X$ as a (labeled) graph $G(X)$ by connecting any two vertices $x, y \in X$ (including $x=y$) with an edge labeled by $b(x,y)$.
Note that this construction is functorial in the sense that if $X \subseteq L$, then any morphism $f \colon L \to L'$ induces a morphism $G(X) \to G(f(X))$ of graphs.

To turn this into an invariant we now consider finite subfunctors $V \colon \LAT \to \SET$ of the forgetful functor $\LAT \to \SET$, that is,
for every lattice $L$ we have a finite subset $V(L) \subseteq L$ such that, in particular, $f(V(L)) = V(L')$ for every isomorphism $f \colon L \to L'$.
We call such a functor $V$ a \textit{vector set functor}.
If in addition $\langle V(L)\rangle = L$ for all lattices $L$, we call $V$ \textit{characteristic}.
Since the graph construction $G$ is functorial, for a vector set functor $V$ we obtain an invariant
\[ G_V \colon \LAT \longrightarrow \GRAPH,\ L \longmapsto G(V(L)). \]
The following lemma shows that characteristic vector set functors can be used to define complete invariants.

\begin{lemma}\label{lem:gvinj}
    If $V \colon \LAT \to \SET$ is a characteristic vector set functor, then $G_V$ is essentially injective.
\end{lemma}

\begin{proof}
   Let $(L, b)$ and $(L', b')$ be lattices.
   Assume that $f \colon G(V(L)) \to G(V(L'))$ is a graph isomorphism. 
   By the choice of edge labels, this induces a map $h \colon V(L) \to V(L')$ such that $b'(h(x), h(y)) = b(x, y)$ for all $x, y \in V(L)$.
   Now suppose that $\sum_{x \in V(L)} a_xx=0$. Then, for every $y\in V(L)$,
\[
  b'\left(\sum_x a_xh(x),h(y)\right)
    = \sum_x a_xb(x,y)=0.
\]
Since $V(L')$ spans $L'$ and $b'$ is non-degenerate, it follows
that $\sum_{x \in V(L)} a_xh(x)=0$. Thus $h$ extends uniquely to a
$\mathbb Z$-linear map $L\to L'$ preserving the bilinear forms.
This map is injective since $L$ is non-degenerate and surjective since $V$ is characteristic.
\end{proof}

\begin{example}
Let $L$ be a lattice.
For $k \geq 0$ denote by $S_k(L)$ the subset of vectors of $L$ of squared length at most $k$, and by $\Smin(L)$ the subset of shortest vectors of $L$. 
Then $L \mapsto \Smin(L)$ and $L \mapsto S_k(L)$
define vector set functors, which are not characteristic.
Denote by $m(L) = \min\{k \in \NN \mid \langle S_k(L) \rangle = L\}$ the smallest squared radius of a ball containing a generating set of $L$.
Then $V_m \colon L \mapsto S_{m(L)}(L)$ defines a characteristic vector set functor $V_{m}$.
In combination with graph canonicalization we can use the essentially injective functor $G_{V_{m}}$ to obtain a complete invariant for lattices.
Since the graphs $G(V_m(L))$ are impractically large, this choice of $V$ is not practical.
\end{example}

\begin{remark}
In~\cite{SHVW} the authors study \textit{characteristic vector sets} of lattices, which correspond to subsets of the form $V(L)$, where $V$ is a characteristic vector set functor.
These characteristic vector sets are used in combination with the graph construction $G$ to define complete invariants of lattices.
To obtain a practical invariant,
in~\cite[\S{}2.2]{SHVW} the authors construct characteristic vector sets $\Vcv(L)$ inductively using closest vectors to the successive sublattices. These characteristic vector sets yield graphs that are often small enough to be of practical use, except when $L$ has a large automorphism group, which forces $\Vcv(L)$ to be large.
\end{remark}

To obtain smaller graphs, we now describe how to incorporate the action of the Weyl group via the functor $\Gamma$ from \Cref{ex:cd-functor}.
First, note that $\Gamma$ does not fit within the framework of characteristic vector sets of~\cite{SHVW} since the fundamental roots of a lattice $L$ are in general only $\Aut(L,\rho)$-invariant, but not $\Aut(L)$-invariant.

\begin{lemma}\label{lemma:span}
Let $L$ be a lattice.
Suppose that $B \subseteq L$ spans $L$ and is $W(L)$-invariant. Let $E$ be a fundamental root system of $L$ and let $P$ be the corresponding Weyl chamber. Then $E \cup (B \cap P)$ spans $L$.
\end{lemma}

\begin{proof}
Since $P$ is a fundamental domain for the action of the Weyl group $W(L)$, and $B$ is $W(L)$-invariant, $B = \bigcup_{w \in W(L)} w(P \cap B)$. Note that $w(b) \equiv b \mod \langle E \rangle$ and so $B$ and $P \cap B$ span the same space modulo $\langle E \rangle$.
\end{proof}

The previous lemma shows that factoring out the action of the Weyl group allows one to reduce the size of the generating sets under consideration.
As before, instead of factoring out the action of the Weyl group we will consider automorphisms fixing a chosen Weyl vector.
Since the choice of a Weyl vector is not canonical, we pass to \textit{Weyl-marked lattices}, which are pairs $(L, \rho)$ consisting of a lattice $L$ together with a Weyl vector $\rho$ (which can be zero if $L$ has no roots).
A morphism from $(L,\rho)$ to $(L',\rho')$ is a morphism $f \colon L \to L'$ of lattices with $f(\rho)=\rho'$. 
By $\WLAT$ we denote the category of Weyl-marked lattices.
Since the Weyl group acts transitively on the chambers, the forgetful functor $\WLAT \to \LAT$ is essentially bijective. This means that $(L,\rho)$ is isomorphic to $(L',\rho')$ if and only if $L\cong L'$.
Note that the automorphism group of $(L,\rho)$ in the category $\WLAT$ is exactly the reduced automorphism group $\Aut(L,\rho)$ of $L$.
Let $P_\rho$ denote the unique Weyl chamber containing $\rho$ and let $E_\rho$ denote the corresponding fundamental root system.

\begin{prop}\label{prop:vcapgamma}
    Let $V\colon \LAT \to \SET$ be a vector set functor.
    It induces functors 
  \[V \colon \WLAT \longrightarrow \SET, \quad (L,\rho)  \longmapsto V(L) \cap P_\rho\] 
and 
    \[G_{V}^{\mathrm{red}} \colon \WLAT \longrightarrow \GRAPH, \ (L,\rho) \longmapsto G((V(L,\rho) \setminus \Phi(L))\cup E_\rho).\]

    If $V$ is characteristic, then the reduced functor $G_{V}^{\mathrm{red}}$ is essentially injective.
\end{prop}
\begin{proof}
If $f\colon (L,\rho) \to (L',\rho')$ is an isomorphism, then $f(V(L)) = V(L')$ and $f(P_\rho)=P_{\rho'}$. Hence $V(L,\rho)$ is well defined. 
Since $f(\Phi(L)) = \Phi(L')$ and $f(E_\rho)=E_{\rho'}$, 
 $f$ restricts to a graph isomorphism
\[G((V(L,\rho) \setminus \Phi(L))\cup E_\rho) \longrightarrow G((V(L',\rho') \setminus \Phi(L'))\cup E_{\rho'})\] 
and the second functor $G_{V}^{\mathrm{red}}$ is well defined. 

Let $(L,\rho)$ and $(L',\rho')$ be two Weyl-marked lattices and let
\[f \colon G_{V}^{\mathrm{red}}(L,\rho) \to G_{V}^{\mathrm{red}}(L',\rho')\]
be a graph isomorphism. Since $\langle \Phi(L)\rangle = \langle E_\rho\rangle$, by \Cref{lemma:span} we have that $(V(L,\rho) \setminus \Phi(L))\cup E_\rho$ generates $L$.
As in the proof of \Cref{lem:gvinj}, the map $f$ extends to an isometry $L\to L'$.
Since $f$ preserves roots and $E_\rho$ is the set of roots of $(V(L,\rho) \setminus \Phi(L))\cup E_\rho$ (and similarly for $L'$), the set $E_\rho$ is mapped to $E_{\rho'}$ and therefore $\rho$ is mapped to $\rho'$ by the linear extension of $f$.
\end{proof}

Combining \Cref{prop:vcapgamma} with the fact that the forgetful functor $\WLAT \to \LAT$ is an essential bijection, we can
construct new complete invariants from any characteristic vector set functor $V$. More precisely, the functor
\[ G_{V}^{\mathrm{red}} \colon \LAT \longrightarrow \GRAPH, \ L \longmapsto \operatorname{Can}\bigl(G((V(L,\rho) \setminus \Phi(L))\cup E_\rho)\bigr), \]
where $\operatorname{Can}$ denotes graph canonicalization and $\rho$ is an arbitrary Weyl vector of $L$,
defines a complete invariant.
Since it is advisable to use a vector set functor $V$ for which the sets $V(L)$ are small, 
in practice we always choose $V = \Vcv$ from~\cite{SHVW}.

\begin{remark}\label{rem:canonical-gram}
The canonical labeling of the graph $G_V^{\mathrm{red}}(L)$ is an ordering of the vertices $v_1,\dots, v_N \in (V(L,\rho)\setminus\Phi(L))\cup E_\rho$ which is unique up to the automorphism group of the graph, that is, up to $\Autred(L)$ (see the proof of \Cref{prop:vcapgamma}).
Following \cite{SHVW} this yields a canonical Gram matrix as follows. Fix some basis $B$ of $L$ and let the rows of $Q \in \ZZ^{N \times r}$ be the coordinate vectors of $v_1,\dots v_N$. Write $Q = HU$ with $H$ in column Hermite normal form and $U \in \GL_{r}(\ZZ)$. Then the canonical Gram matrix of $L$ is defined as $\Gcan := UG_BU^T$. 
A different canonical ordering changes $Q$ to $QF$ for some $F \in \Autred(L)$ and a different choice of basis $C = XB$ changes $G_B$ to $XG_BX^T$, $U$ to $UX^{-1}$.
This does not change $\Gcan$ and the matrix is indeed canonical. 
\end{remark}

\begin{remark}
The main computational advantage of $G_{V}^{\mathrm{red}}$ is that its
vertex set is typically much smaller than $V(L)$, so that graph
canonicalization is applied to a smaller graph. A naive implementation
first computes $V(L)$ and then forms
\[
  ((V(L)\setminus\Phi(L))\cap P_\rho)\cup E_\rho.
\]
Whenever the construction of $V$ permits, one can instead compute
$V(L,\rho)$ directly, thereby also reducing the intermediate computational effort. 
\end{remark}

\begin{remark}
  Using the result of \S{}\ref{subsec:close-vectors}, we can improve
  the computation of $\Vcv(L, \rho)$ as follows.
  We briefly recall the definition of $\Vcv$.
  For a positive definite integer lattice $M$ and $x \in M \otimes \RR$, denote by $\cv(M, x)$ the set of vectors of $M$ closest to $x$
  and for a subset $X \subseteq M$, denote by $\sv(X)$ the set of its shortest vectors.

  The set $\Vcv(L)$ is defined recursively as follows:
  For a well-rounded lattice $L$, define
  \[\Vcv(L) = \Smin(L) \cup \bigcup_{x+L_1 \in L/L_1} \sv(x+L_1), \]
  where $L_1 = \langle \Smin(L) \rangle$ is the sublattice of $L$ generated by its shortest vectors.
  If $L$ is not well-rounded, let $\pi$ denote the orthogonal projection from $L$ to $L_1^\perp \otimes \QQ$ and $\overline{L}_1$ the primitive closure of $L_1$ in $L$. We define
  \[ \Vcv(L) = \Vcv(\overline{L}_1) \cup \bigcup_{x \in \Vcv(\pi(L))} \sv(\pi^{-1}(\{x\})).\]
  We have $\sv(x+L_1) = \{x-y \mid y \in \cv(L_1,x)\}$
  and for $x=\pi(x') \in \pi(L)$ we have
  $\sv(\pi^{-1}(\{x\})) = \sv(x'+\overline{L}_1)$. Thus, for both sets, the computation of shortest vectors reduces to a closest vector problem where the target $x$ is in $L_1^\vee$ and the lattice is $L_1$ or $\overline{L}_1$, respectively.
  
  If $L_1$ is a root lattice and we further intersect with a Weyl chamber $P_\rho$ to compute the reduced characteristic vector set $\Vcv(L,\rho)$, then \Cref{lem:close_vector_root} reduces closest vector problems at the first level of recursion, in $L_1\cap P_\rho$, to a simple lookup.
  To compute closest vectors to $y \in L_1^\vee$ in $\overline{L}_1$, decompose the coset into $L_1$-cosets
  \[y+ \overline{L}_1=\bigsqcup_{x + L_1 \in \overline{L}_1/L_1} (y+x) + L_1.\]
  Then $\sv(y + \overline{L}_1)$ is the set of shortest vectors in the finite set $\bigcup_{x + L_1 \in \overline{L}_1/L_1} \sv((y+x) + L_1)$. Since we are only interested in solutions in the Weyl chamber, we form the subset
  \[\sv(P_\rho \cap (y + \overline{L}_1)) \subseteq \bigcup_{x + L_1 \in \overline{L}_1/L_1} \sv(P_\rho \cap (y+x + L_1)).\]
\end{remark}

\section{Experiments}\label{sec:exp}

We implemented the improved Plesken--Souvignier algorithm in \textsc{Hecke}~\cite{Hecke}, part of the open-source computer algebra system \textsc{OSCAR}~\cite{OSCARBook,OSCAR}.
To benchmark the improvements, we have run both the improved Plesken--Souvignier algorithm and the complete lattice invariant computation on a large dataset of lattices.

\subsection{Dataset}\label{subsec:datasets}

We compiled a publicly available collection of lattices~\cite{LatticeAutBench}.
Our experiments use 13\,083 lattices of ranks $17$ to $26$, grouped into the following three subfamilies:

\begin{enumerate}
  \item[(A)]
  10\,814 lattices of rank 17, the genus representatives of
  $\II_{(17, 0)}{2}^{1}_{1}{9}^{-1}{5}^{-1}$, $\II_{(17, 0)}[{2}^{2}{4}^{1}]_{3}{7}^{-1}$ and $\II_{(17, 0)}{2}^{3}_{7}{3}^{1}{5}^{1}$,
  \item[(B)]
  368 lattices of rank 25: the odd unimodular lattices of rank 25 with no norm-$1$ vectors,
  \item[(C)]
    1\,901 lattices of rank 26: the odd unimodular lattices of rank 26 with no norm-$1$ vectors.
\end{enumerate}

The dataset in (A) has been compiled by the first-named author during investigations on K3 surfaces without $\FF_q$-rational points.
For (B) and (C), we use representatives provided by Chenevier--Ta\"ibi; see \cite{ChenevierTabi2026}.
The distributions of the ranks of the root sublattices are shown in
\Cref{fig:root-rank-histogram}. To quantify lopsidedness, we consider the
parameter
\[
  \lambda(L)
  = \frac{\min(L)}
  {\min\{k \in \ZZ_{\geq 0} \mid
    \langle S_k(L)\rangle_{\QQ}=L\otimes\QQ\}}.
\]
It satisfies $0<\lambda(L)\leq 1$, with equality if and only if $L$ is
well-rounded. Thus $L$ is lopsided if and only if $\lambda(L)<1$.
The distributions of $\lambda(L)$ are shown in
\Cref{fig:lopsidedness-distributions}.

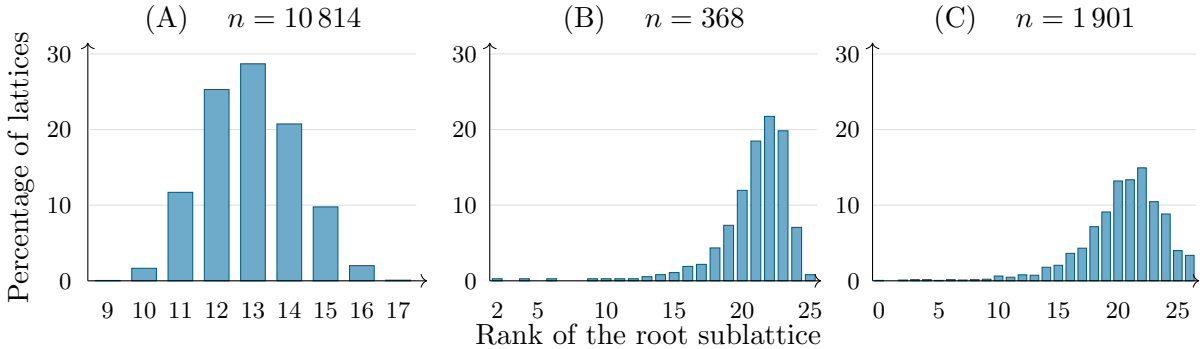
\begin{figure}[htbp]
  \centering
  \begin{tikzpicture}[y=0.1cm]
    \begin{scope}[x=0.48cm]
      \foreach \height/\ticklabel in {0/0,10/10,20/20,30/30}
      {
        \draw[gray!25] (-0.5,\height) -- (8.5,\height);
        \node[anchor=east,font=\scriptsize] at (-0.7,\height) {\ticklabel};
      }
      \foreach \rank/\count
        [evaluate=\rank as \position using \rank-9] in {
          9/5,10/179,11/1265,12/2736,13/3103,14/2243,15/1057,
          16/217,17/9}
      {
        \pgfmathsetmacro{\percentage}{\count/10814*100}
        \filldraw[fill=MidnightBlue!45,draw=MidnightBlue!80!black]
          ({\position-0.34},0) rectangle ({\position+0.34},\percentage);
      }
      \foreach \rank [evaluate=\rank as \position using \rank-9] in {9,...,17}
        \node[anchor=north,font=\scriptsize] at (\position,-1.4) {\rank};
      \draw[->] (-0.55,0) -- (8.8,0);
      \draw[->] (-0.55,0) -- (-0.55,31.5);
      \node[font=\small] at (4,34.5) {(A) \quad $n=10\,814$};
    \end{scope}

    \begin{scope}[xshift=5.15cm,x=0.18cm]
      \foreach \height/\ticklabel in {0/0,10/10,20/20,30/30}
      {
        \draw[gray!25] (-0.5,\height) -- (23.5,\height);
        \node[anchor=east,font=\scriptsize] at (-0.7,\height) {\ticklabel};
      }
      \foreach \rank/\count
        [evaluate=\rank as \position using \rank-2] in {
          2/1,4/1,6/1,9/1,10/1,11/1,12/1,13/2,14/3,15/4,
          16/7,17/8,18/16,19/27,20/44,21/68,22/80,23/73,24/26,25/3}
      {
        \pgfmathsetmacro{\percentage}{\count/368*100}
        \filldraw[fill=MidnightBlue!45,draw=MidnightBlue!80!black]
          ({\position-0.36},0) rectangle ({\position+0.36},\percentage);
      }
      \foreach \rank [evaluate=\rank as \position using \rank-2] in {
        2,5,10,15,20,25}
        \node[anchor=north,font=\scriptsize] at (\position,-1.4) {\rank};
      \draw[->] (-0.55,0) -- (23.8,0);
      \draw[->] (-0.55,0) -- (-0.55,31.5);
      \node[font=\small] at (11.5,34.5) {(B) \quad $n=368$};
    \end{scope}

    \begin{scope}[xshift=10.2cm,x=0.158cm]
      \foreach \height/\ticklabel in {0/0,10/10,20/20,30/30}
      {
        \draw[gray!25] (-0.5,\height) -- (26.5,\height);
        \node[anchor=east,font=\scriptsize] at (-0.7,\height) {\ticklabel};
      }
      \foreach \rank/\count in {
          0/1,2/2,3/3,4/3,5/1,6/3,7/2,8/3,9/4,10/12,11/9,12/15,
          13/14,14/34,15/39,16/69,17/82,18/136,19/173,20/251,
          21/254,22/284,23/199,24/168,25/76,26/64}
      {
        \pgfmathsetmacro{\percentage}{\count/1901*100}
        \filldraw[fill=MidnightBlue!45,draw=MidnightBlue!80!black]
          ({\rank-0.36},0) rectangle ({\rank+0.36},\percentage);
      }
      \foreach \rank in {0,5,10,15,20,25}
        \node[anchor=north,font=\scriptsize] at (\rank,-1.4) {\rank};
      \draw[->] (-0.55,0) -- (26.8,0);
      \draw[->] (-0.55,0) -- (-0.55,31.5);
      \node[font=\small] at (13,34.5) {(C) \quad $n=1\,901$};
    \end{scope}

    \node[rotate=90,font=\small] at (-1.15,15) {Percentage of lattices};
    \node[font=\small] at (7.15,-7) {Rank of the root sublattice};
  \end{tikzpicture}
  \caption{Distributions of the ranks of the root sublattices in
    datasets~(A)--(C). Bar heights give percentages within each dataset.}
  \label{fig:root-rank-histogram}
\end{figure}

\begin{figure}[htbp]
  \centering
  \begin{tikzpicture}[x=0.1cm,y=0.48cm]
    \foreach \percentage in {0,25,50,75,100}
    {
      \draw[gray!25] (\percentage,-0.5) -- (\percentage,7.5);
      \node[anchor=north,font=\scriptsize] at (\percentage,-0.65)
        {\percentage};
    }
    \foreach \value/\percentage [count=\position from 0] in {
      {\frac{1}{15}}/0.10,{\frac{1}{9}}/0.05,{\frac{1}{6}}/0.18,
      {\frac{1}{5}}/0.64,{\frac{1}{4}}/0.70,{\frac{1}{3}}/4.04,
      {\frac{1}{2}}/94.21,1/0.08}
    {
      \filldraw[fill=MidnightBlue!45,draw=MidnightBlue!80!black]
        (0,{\position-0.32}) rectangle (\percentage,{\position+0.32});
      \node[anchor=east,font=\scriptsize] at (-2,\position) {$\value$};
      \node[anchor=west,font=\scriptsize] at ({\percentage+1.2},\position)
        {\percentage\%};
    }
    \draw[->] (0,-0.5) -- (102,-0.5);
    \node[font=\small] at (50,8.4) {(A) \quad $n=10\,814$};
    \node[anchor=north,font=\small] at (50,-1.45)
      {Percentage of lattices};
  \end{tikzpicture}

  \medskip
  \begin{minipage}{0.47\textwidth}
    \centering
    \begin{tikzpicture}[x=0.045cm,y=0.55cm]
      \foreach \percentage in {0,25,50,75,100}
      {
        \draw[gray!25] (\percentage,-0.5) -- (\percentage,1.5);
        \node[anchor=north,font=\scriptsize] at (\percentage,-0.65)
          {\percentage};
      }
      \foreach \value/\percentage [count=\position from 0] in {
        {\frac{2}{3}}/99.18,1/0.82}
      {
        \filldraw[fill=MidnightBlue!45,draw=MidnightBlue!80!black]
          (0,{\position-0.32}) rectangle (\percentage,{\position+0.32});
        \node[anchor=east,font=\scriptsize] at (-5,\position) {$\value$};
        \node[anchor=west,font=\scriptsize] at ({\percentage+2.5},\position)
          {\percentage\%};
      }
      \draw[->] (0,-0.5) -- (104,-0.5);
      \node[font=\small] at (50,2.4) {(B) \quad $n=368$};
      \node[anchor=north,font=\small] at (50,-1.55)
        {Percentage of lattices};
    \end{tikzpicture}
  \end{minipage}%
  \hfill
  \begin{minipage}{0.47\textwidth}
    \centering
    \begin{tikzpicture}[x=0.045cm,y=0.55cm]
      \foreach \percentage in {0,25,50,75,100}
      {
        \draw[gray!25] (\percentage,-0.5) -- (\percentage,1.5);
        \node[anchor=north,font=\scriptsize] at (\percentage,-0.65)
          {\percentage};
      }
      \foreach \value/\percentage [count=\position from 0] in {
        {\frac{2}{3}}/96.58,1/3.42}
      {
        \filldraw[fill=MidnightBlue!45,draw=MidnightBlue!80!black]
          (0,{\position-0.32}) rectangle (\percentage,{\position+0.32});
        \node[anchor=east,font=\scriptsize] at (-5,\position) {$\value$};
        \node[anchor=west,font=\scriptsize] at ({\percentage+2.5},\position)
          {\percentage\%};
      }
      \draw[->] (0,-0.5) -- (104,-0.5);
      \node[font=\small] at (50,2.4) {(C) \quad $n=1\,901$};
      \node[anchor=north,font=\small] at (50,-1.55)
        {Percentage of lattices};
    \end{tikzpicture}
  \end{minipage}
  \caption{Distributions of $\lambda(L)$ in datasets~(A)--(C).
    The values are shown as exact fractions.}
  \label{fig:lopsidedness-distributions}
\end{figure}
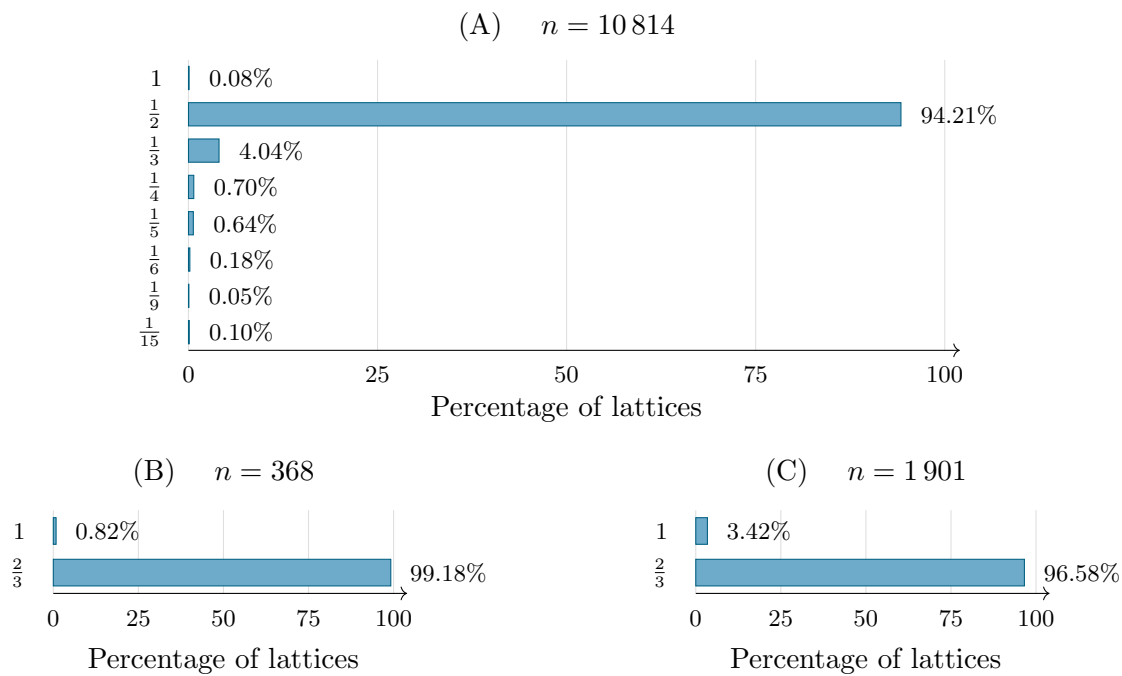

\subsection{Automorphism group computation}

We benchmarked three implementations on datasets (A)--(C): our new algorithm in \textsc{Hecke} version 0.40.4 and Julia version 1.12.0,
Ta\"ibi's \texttt{qfautors} function~\cite{ChTa} implemented in PARI/GP version 2.18.0, and Magma's
\texttt{Automorphism\-Group} function (version V2.28-8).
The \texttt{depth} and \texttt{Bacher-depth} parameters are zero by default for \texttt{New} and for \texttt{qfautors}. Magma ran with default parameters.
For each implementation and dataset,
\Cref{tab:automorphism-group-timings} reports the mean, median, and maximum
running times in seconds for single-threaded computations. The benchmarks were obtained on a server with an Intel
Xeon E5-2643 v3 CPU clocked at 3.40~GHz and 378~GiB of RAM. Each computation was subject to a time limit of 300 minutes
and a memory limit of 32~GB. Magma timed out for 16 lattices in dataset (A), but finished for all other lattices.

The maximum runtimes in \Cref{tab:automorphism-group-timings} confirm that the Plesken--Souvignier algorithm struggles with lopsided lattices. As the maximum runtime for dataset (A) shows, this obstruction is lifted by the new algorithm.
Its median runtimes are comparable to those of \texttt{qfautors} and lower than those of \texttt{Magma}, while it outperforms its competitors on average in the present datasets (A)--(C).
However, for perfectly well-rounded lattices without roots such as the Leech lattice, the new algorithm does not contribute anything and essentially coincides with the standard Plesken--Souvignier algorithm.

\subsection{Characteristic vector sets and canonical forms}\label{subsec:inv-exp}

We have also implemented the computation of $\Vcv(L,\rho)$ and the canonical form.
The latter uses the graph
functionality in \textsc{OSCAR}, which delegates these tasks to
\textsc{nauty}~\cite{MR3131381}. For each dataset,
\Cref{tab:characteristic-vector-set-data} reports the mean running time to
compute $\Vcv(L,\rho)$, the mean total running time for the canonical form computation
(which includes the computation of $\Vcv(L, \rho)$), the mean cardinality of
$\Vcv(L,\rho)$, and the mean reduction factor
\[
  \frac{\lvert \Vcv(L) \rvert}{\lvert \Vcv(L,\rho) \rvert}.
\]
The last two columns have been rounded to the nearest integer.

\begin{table}[htbp]
  \caption{Mean running times and vector set statistics for the computation
    of canonical forms on datasets~(A)--(C). Times are given in seconds.}
  \label{tab:characteristic-vector-set-data}

  \centering
  \setlength{\tabcolsep}{8pt}
  \begin{tabular}{lrrrr}
    \toprule
    & \multicolumn{2}{c}{Mean time}
    & \multicolumn{2}{c}{Mean vector set statistics} \\
    \cmidrule(lr){2-3}
    \cmidrule(lr){4-5}
    Dataset
      & $\Vcv(L,\rho)$
      & Canonical form
      & $\#\Vcv(L,\rho)$
      & $\#\Vcv(L)/\#\Vcv(L,\rho)$ \\
    \midrule
    (A) & 0.002 & 0.004 &  23 &  102 \\
    (B) & 0.008 & 0.047 & 116 & 3598 \\
    (C) & 0.011 & 0.055 & 162 & 6303 \\
    \bottomrule
  \end{tabular}
\end{table}
Canonicalization with nonreduced vector sets as in \cite{SHVW} is infeasible on the present datasets (A)--(C). Since the reduced vector sets are orders of magnitude smaller, canonicalization of the lattices in (A)--(C) is now possible in a reasonable running time.
The timings show that the main bottleneck has shifted to the graph-theoretic
part of the canonical form computation, particularly for datasets~(B) and~(C), where
computing the characteristic vector set accounts for only about one fifth
of the total running time. Nevertheless, the total canonical form computation remains
of the same order of magnitude as the automorphism group computation with
our algorithm; see \Cref{tab:automorphism-group-timings}. Thus a canonical
form can be computed at a cost comparable to that of computing the
automorphism group.

Restricting to the Weyl chamber substantially reduces the characteristic
vector sets. The mean reduction factor is about $10^2$ for dataset~(A), and
increases to approximately $3.6\cdot 10^3$ and $6.3\cdot 10^3$ for
datasets~(B) and~(C), respectively. These larger reductions are consistent
with the greater root-sublattice ranks visible in
\Cref{fig:root-rank-histogram}.

\bibliography{main.bib}
\bibliographystyle{amsalpha}

\end{document}